\documentclass[11pt, reqno]{amsart}
\usepackage{amsmath,amsfonts,amsthm,amssymb,latexsym}
\usepackage[a4paper,margin=3cm]{geometry}
\usepackage{graphicx}
\usepackage{url}
\usepackage[latin1]{inputenc}
\usepackage[T1]{fontenc} 
\usepackage{graphicx}
\usepackage[english]{babel}
\usepackage{calc}
\usepackage{amsmath,amsthm}
\usepackage{url}
\usepackage{times, amssymb, amscd, mathrsfs, graphicx, color}  
\usepackage{enumerate}
\usepackage{hyperref}
\usepackage{dsfont}
\usepackage[all]{xy}
\definecolor{darkred}{rgb}{0.9,0.1,0.1}

\newtheorem{theo}{Theorem}[section]

\newtheorem{coro}[theo]{Corollary}
\newtheorem{lem}[theo]{Lemma}

\newtheorem{Rq}[theo]{Remark}

\theoremstyle{plain}
\newtheorem*{hypo*}{Assumption}

\newcommand{\un}{\mathbf{1}} 
\newcommand{\N}{\mathbb{N}}                                              
\newcommand{\R}{\mathbb{R}}                                              

\newcommand{\E}{\mathbb{E}}                                            

\newcommand{\var}{\textrm{Var}}  
\newcommand{\cov}{\textrm{Cov}}

\title{Fleming-Viot processes : two explicit examples}

\date{\today}

\author{Bertrand~\textsc{Cloez}}  %
\address[Bertrand~\textsc{Cloez}]{umr inra-supagro MISTEA, Montpellier, France}%
\email{\url{bertrand.cloez@supagro-inra.fr}}%

\author{Marie-No\'{e}mie~\textsc{Thai}} %
\address[Marie-No\'{e}mie~\textsc{Thai}]{CEREMADE, Universit\'e Paris-Dauphine,
  France} %
\email{\url{noemiethai@yahoo.fr}}

\usepackage{float}

\begin{document}

\maketitle
\begin{abstract} 
The purpose of this paper is to extend the investigation of the Fleming-Viot process in discrete space started in a previous work to two specific examples. The first one corresponds to a random walk on the complete graph. Due to its geometry, we establish several explicit and optimal formulas for the Fleming-Viot process (invariant distribution, correlations, spectral gap). The second example corresponds to a Markov chain in a two state space. In this case, the study of the Fleming-Viot particle system is reduced to the study of birth and death process with quadratic rates.
\end{abstract} 
{\footnotesize \textbf{AMS 2000 Mathematical Subject Classification:} 60K35, 60B10, 37A25.}

{\footnotesize \textbf{Keywords:} Fleming-Viot process - quasi-stationary distributions - coupling - Wasserstein distance - chaos propagation - commutation relation.}

\renewcommand*\abstractname{ \ }
\begin{abstract}
\tableofcontents
\end{abstract}

\section{Introduction}
In discrete space, the Fleming-Viot particle system has been studied by many authors \cite{AFG,AFGJ,AT12,FM07,Maric14}. Such a system is a mean field particle system described in the following way: we consider $N$ copies of an absorbed Markov chain and, instead of being absorbed, one chain jumps randomly on the state of another one. It is well known that, when the number of copies tends to infinity, the empirical measure converges to the law of the initial chain conditioned not to be absorbed, see for instance \cite{CT15,MM00,GJ12,V11}. Convergence to equilibrium as time goes to infinity is less known. In \cite{MR06,V10}, this question is addressed for some models. Nevertheless, to our knowledge, there are few results on the expression of the invariant distribution or on the explicit rates of convergence. This paper is concerned with studying two specific models for which the invariant distribution is explicit and extend the investigation started by Cloez and Thai \cite{CT15}.\\
Let $Q=\left(Q_{i,j}; \ i,j \in F^* \cup \{ 0 \}\right)$ be the transition rates matrix of an irreducible and positive recurrent continuous time Markov process $(X_t)_{t\geq 0}$ on a countable state space $F=F^* \cup \{ 0 \}$. We think of $0$ as an absorbing state. Let $\mu$ be the initial law of $(X_t)_{t\geq 0}$ and let $\mu T_t$ be its law at time $t$ conditioned on non absorption up to time $t$. That is defined, for all non-negative function $f$ on $F^*$, by
$$
\mu T_t f = \frac{\mu P_t f}{ \mu P_t \un_{\{0\}^c} } = \frac{\sum_{y \in F^{*}} P_t f(y) \mu(y)}{\sum_{y \in F^{*}} P_t \un_{\{0\}^c}(y) \mu(y)},
$$
where $(P_t)_{t\geq 0}$ is the semigroup associated with the transition matrix $Q$ and we use the convention $f(0)=0$. For every $x\in F^*$, $k\in F^*$ and  non-negative function $f$ on $F^*$, we also set 
$$T_tf(x) = \delta_x T_t f \ \text{ and } \ \mu T_t (k) = \mu T_t \un_{\{k\}}, \quad \forall t\geq 0.
$$
A quasi-stationary distribution (QSD) for $Q$ is a probability measure $\nu_{\textrm{qs}}$ on $F^*$ satisfying, for every $t \geq 0$, $ \nu_{\textrm{qs}} T_t = \nu_{\textrm{qs}}$.

The particle system we are focusing on was initially introduced in \cite{DMG,MM00} for approximating the conditioned semigroup $(T_t)_{t\geq 0}$ and the QSD $\nu_{qs}$.
It is convenient to think of particles as being indistinguishable, and to consider the occupation number $\eta $ with, for $k \in F^*$ , $\eta(k)=\eta^{(N)}(k)$ representing the number of particles at site $k$. The configuration $(\eta_t)_{t\geq0}$ is a Markov process with state space $E=E^{(N)}$ defined by
$$
E=\left\{ \eta : F^* \to \N \ | \ \sum_{i \in F^*}  \eta(i) =N \right\}.
$$
Applying its generator to a bounded function $f$ gives
\begin{equation}
\label{eq:generator2}
\mathcal{L}f(\eta)=\mathcal{L}^{(N)}f(\eta) = \sum_{i\in F^{*}} \eta(i) \left[ \sum_{j \in F^{*}} (f(T_{i \rightarrow j} \eta)-f(\eta)) \left(Q_{i,j} + Q_{i,0} \dfrac{\eta(j)}{N-1} \right)\right],
\end{equation}
for every $\eta \in E$, where, if $\eta(i) \neq 0$, the configuration $T_{i \rightarrow j} \eta$ is defined by 
$$
T_{i \rightarrow j} \eta(i) = \eta(i)-1, \ T_{i \rightarrow j} \eta(j) = \eta(j)+1, \ \text{ and} \ T_{i \rightarrow j} \eta(k) = \eta(k) \quad k \notin \{i,j\}.
$$

The present paper is a continuation of \cite{CT15} in which the following limits are studied and quantified:
\begin{figure}[hbtp]
$$
\xymatrix{
    & \mu_t^{N} \ar[rd]^{t \rightarrow \infty} \ar[ld]_{N \rightarrow \infty} & \\
	\mu T_t \ar[rd]_{t \rightarrow \infty} &  & \mu_{\infty}^{N} \ar[ld]^{N \rightarrow \infty} \\
	& \nu_{qs} &
}
$$  
\end{figure}

where $\mu^{N}$ is the associated empirical distribution of the particle system defined, for $\eta \in E$, by
$$
\mu_t^{N}= \dfrac{1}{N}\sum_{k\in F^{*}} \eta(k) \delta_{\{k\}}.
$$
For countable space $F$, the ergodicity of the Fleming-Viot process is not guaranteed. In \cite{CT15}, Cloez and Thai show that under some conditions, the particle system converges exponentially fast to equilibrium for a suitable Wasserstein coupling distance. Let us recall the different distances given by the autors. For $\eta, \eta'\in E$, let $d$ be the distance defined by
\begin{equation}
d(\eta,\eta') = \frac{1}{2} \sum_{j\in F} \vert \eta(j) - \eta'(j)\vert,
\end{equation}
and for any two probability measures $\mu$ and $\mu'$ on $E$, let $\mathcal{W}_d(\mu, \mu')$ be the Wasserstein coupling distance between these two laws defined by
\begin{equation}
\mathcal{W}_d(\mu, \mu') = \inf_{\substack{X \sim \mu\\ \overline{X} \sim \mu'}} \E\left[d(X,\overline{X})\right],
\end{equation}
where the infimum runs over all the couples of random variables with marginal laws $\mu$ and $\mu'$. 
\begin{theo}[Theorem 1.1 of \cite{CT15}]
\label{th:tl-general}
Let $\lambda = \displaystyle \inf_{i,i'\in F^{*} }\left( Q_{i,i'} + Q_{i',i} + \sum_{j \neq i,i'} Q_{i,j} \wedge Q_{i',j} \right)$ and for $i \in F^{*}$,  $p_0(i)= Q_{i,0}$. If $\rho= \lambda -(\sup(p_0)-\inf(p_0))$ then for any processes $(\eta_{t})_{t>0}$ and $(\eta'_{t})_{t>0}$ generated by \eqref{eq:generator2}, and for any $t \geq 0$, we have
$$
\mathcal{W}_{d} (\mathrm{Law}(\eta_t), \mathrm{Law}(\eta'_t)) \leq e^{-\rho t} \mathcal{W}_{d} (\mathrm{Law}(\eta_0), \mathrm{Law}(\eta'_0)).
$$
In particular, if $\rho>0$ then there exists a unique invariant distribution $\nu_{N}$ satisfying for every $t \geq 0$,
$$
\mathcal{W}_{d} (\mathrm{Law}(\eta_t), \nu_{N}) \leq e^{-\rho t} \mathcal{W}_{d} (\mathrm{Law}(\eta_0), \nu_{N}).
$$
\end{theo}

\bigskip

This theorem gives the existence and uniqueness of the invariant distribution, but this one is not explicit. Actually, there are few models for which an explicit formula of the invariant distribution is given. This observation is one of the motivations of the present paper.
So, it is interesting to consider and study a model satisfying the previous point : the example of random walk on a complete graph. An interesting point of the complete graph approach is that it permits to reduce the difficulties of the Fleming-Viot to the interaction. Due to its simple geometry, several explicit formulas are obtained such as the invariant distribution, the correlations and the spectral gap. It seems to be new in the context of Fleming-Viot particle systems. A second model for which the invariant distribution is explicit is the two point case, the study of the particle system is then reduced to a birth and death process with quadratic rates. The bound obtained in Theorem \ref{th:tl-general} is not optimal. Nevertheless, the coupling introduced in \cite{CT15} in order to prove Theorem \ref{th:tl-general}, permits us to obtain the spectral gap as rate of convergence. Moreover, we show that the spectral gap of the Fleming-Viot process is always bounded from below by a positive constant not depending on the number of particles.

\bigskip

The remainder of the paper is as follows. Section \ref{sect:GC} is dedicated to the study of random walk on the complete graph and Section \ref{sect:two} to that of the two point case.
\section{Complete graph dynamics}
\label{sect:GC}
In all this section, we study the example of a random walk on the complete graph. Let us fix $K\in \N^*$, $p>0$ and $N\in \N^*$, the dynamics of this example is as follows: we consider a model with $N$ particles and $K+1$ vertices $0,1,\dots, K$. The $N$ particles move on the $K$ vertices $1,\dots, K$ uniformly at random and jump to $0$ with rate $p$. When a particle reaches the node $0$, it jumps instantaneously over another particle chosen uniformly at random. This particle system corresponds to the model previously cited with parameters 
$$
Q_{i,j} = \frac{1}{K}, \quad \forall i,j \in F^{*}=\{1,\dots, K\}, i\neq j \ \text{and} \ Q_{i,0} = p, \quad \forall i\in F^{*}.
$$
The generator of the associated Fleming-Viot process is then given by
\begin{equation}
\label{eq:generator}
\mathcal{L}f(\eta) = \sum_{i=1}^{K} \eta(i) \left[  \sum_{j=1}^{K} (f(T_{i\rightarrow j} \eta)-f(\eta))\left( \dfrac{1}{K}+ p \dfrac{\eta(j)}{N-1} \right)\right],
\end{equation}
for every function $f$ and $\eta \in E$.

A process generated by \eqref{eq:generator} is an instance of inclusion processes studied in \cite{GRV10,GRV11,GRV12}. It is then related to models of heat conduction. One main point of \cite{GRV10,GRV11} is a criterion ensuring the existence and reversibility of an invariant distribution for the inclusion processes. In particular, they give an explicit formula of the invariant distribution of a process generated by \eqref{eq:generator}
 and we give this expression in Subsection \ref{subsec:invariant}. They also study different scaling limits which seem to be irrelevant for our problems.

Another application of this example comes from population genetics. Indeed, this model can also be referred as \textit{neutral evolution}, see for instance \cite{E12,W76}. More precisely, consider $N$ individuals possessing one type in $F^*=\{1, \dots, K\}$ at time $t$. Each pair of individuals
interacts at rate $p$. Upon an interacting event, one individual dies and the other one
reproduces. In addition, every individual changes its type (mutates) at rate $1$ and chooses uniformly at random a new
type in $F^*$. The measure $\mu_t^{N}$ gives the proportions of types. The kind of mutation we consider here is often referred as parent-independent or the house-of-cards model.

In all this section, for any probability measure $\mu$ on $E$, we set in a classical manner $\E_{\mu}[ \cdot ] = \displaystyle \int_{F^*} \E_x[ \cdot ] \mu(dx)$ and $\mathbb{P}_{\mu} = \E_\mu[ \mathds{1}_{\cdot}]$; similarly $\cov_\mu$ and $\var_\mu$ are defined with respect to $\E_\mu$.  

\subsection{The associated killed process}
We define the process $(X_t)_{t\geq 0}$ by setting
$$
X_t = \left\{
    \begin{array}{ll}
        Z_t & \mbox{if } t < \tau \\
        0 & \mbox{if } t \geq  \tau,
    \end{array}
\right.
$$
where $\tau$ is an exponential variable with mean $1/p$ and $(Z_t)_{t\geq 0}$ is the classical complete graph random walk (i.e. without extinction) on $\{1,\dots, K\}$. We have, for any bounded function $f$, 
$$
T_t f(x) = \E\left[ f(X_t) \ | \ X_0=x, X_t \neq 0 \right], \quad t\geq 0, x\in F^*.
$$
The conditional distribution of $X_t$ is simply given by the distribution of $Z_t$ :
$$
                         \mathbb{P} (X_t = i \ | \ X_t \neq 0) = \mathbb{P} (Z_t = i).
$$
The study of $(Z_t)_{t\geq 0}$ is trivial. Indeed, it converges exponentially fast to the uniform distribution $\pi_K$ on $\{1,\dots, K\}$. We deduce that for all $t  \geq 0$ and all initial distribution $\mu$, 
$$
d_{\textrm{TV}}(\mu T_{t}, \pi_{K}) = \sum_{i=1}^K \left|\mathbb{P}_{\mu} (X_t = i \ | \ \tau > t) - \pi_{K}(i)\right| \leq e^{- t}. 
$$
 
Thus in this case, the conditional distribution of $X$ converges exponentially fast to the Yaglom limit $\pi_K$.

\subsection{Correlations at fixed time}

The special form of $\mathcal{L}$, defined at \eqref{eq:generator}, makes the calculation of the two-particle correlations at fixed time easy. 
\begin{theo}[Two-particle correlations]
\label{th:cor-GC}
For all $k,l \in \lbrace 1, \dots , K\rbrace$, $k\neq l$ and any probability measure $\mu$ on $E$, we have for all $t\geq 0$
\begin{align*}
\cov_{\mu}(\eta_{t}(k),\eta_{t}(l))
& = \E_{\mu}\left[ \eta_{0}(k)\eta_{0}(l)\right] e^{-\frac{2K(N-1+p)}{K(N-1)}t} \\
&\quad + \dfrac{-N+1+2pN}{K(N-1+2p)} (\E_{\mu}\left[ \eta_{0}(k)\right]+\E_{\mu}\left[ \eta_{0}(l)\right])e^{-t}\\
&\quad - \E_{\mu}\left[ \eta_{0}(k)\right]\E_{\mu}\left[ \eta_{0}(l)\right]e^{-2t} + \dfrac{-N^{2}(p+1)+N}{K^{2}(N-1+p)}.
\end{align*}
\end{theo}

\begin{Rq}[Limit $t\rightarrow + \infty$]
By the previous theorem, we find for any probability measure $\mu$
\begin{align*}
\lim_{t\rightarrow +\infty} \cov_{\mu}(\eta_{t}(k),\eta_{t}(l)) 
&= \dfrac{-N^{2}(p+1)+N}{K^{2}(N-1+p)}= \cov(\eta(k),\eta(l)),
\end{align*}
where $\eta$ is distributed according to the invariant distribution; it exists since the state space is finite, see the next section.
\end{Rq}

\begin{Rq}[Limit $N\rightarrow + \infty$]

If  $\cov_{\mu}\left(\eta_{0}(k),\eta_{0}(l)\right)\neq 0$ then for all $k,l \in \lbrace 1, \dots , K\rbrace$, $k\neq l$ and any probability measure $\mu$, we have
\begin{align*}
 \cov_{\mu}\left(\dfrac{\eta_{t}(k)}{N},\dfrac{\eta_{t}(l)}{N}\right)
&\sim_{N}  e^{-2t } \cov_{\mu}\left(\dfrac{\eta_{0}(k)}{N},\dfrac{\eta_{0}(l)}{N}\right),
\end{align*}
where $u_{N}\sim_{N} v_{N}$ iff $\lim \limits_{N\rightarrow +\infty} \dfrac{u_{N}}{v_{N}} = 1$.
\end{Rq}

\begin{proof}[Proof of Theorem \ref{th:cor-GC}]
For $k,l \in \{1,..,K\}$, let $\psi_{k,l}$ be the function $\eta\mapsto \eta(k)\eta(l)$. Applying the generator \eqref{eq:generator} to $\psi_{k,l}$ we obtain 
\begin{align*}
\mathcal{L}\psi_{k,l}(\eta)&= -\dfrac{2K(N-1+p)}{K(N-1)} \eta(k)\eta(l) + \dfrac{N-1}{K}(\eta(k)+\eta(l)).
\end{align*}
So, for all $t\geq0$,
$$\mathcal{L}\psi_{k,l}(\eta_{t}) = -\dfrac{2K(N-1+p)}{K(N-1)} \eta_{t}(k)\eta_{t}(l) + \dfrac{N-1}{K}(\eta_{t}(k)+\eta_{t}(l)).
$$
Using Kolmogorov's equation, we have
\begin{equation}
\label{eq:EDO}
\partial_{t} \E_{\mu}(\eta_{t}(k)\eta_{t}(l)) = -\dfrac{2K(N-1+p)}{K(N-1)} \E_{\mu}(\eta_{t}(k)\eta_{t}(l)) + \dfrac{N-1}{K}(\E_{\mu}(\eta_{t}(k))+\E_{\mu}(\eta_{t}(l))).
\end{equation}
Now if $\varphi_{k}(\eta) =\eta(k)$ then $\mathcal{L}\varphi_{k}(\eta) = \dfrac{N}{K} - \eta(k)$. We deduce that, for every $t\geq 0$,
$$
\partial_{t} \E_{\mu}(\eta_{t}(k)) = \dfrac{N}{K} - \E_{\mu}(\eta_{t}(k)) \ \text{ and } \ \E_{\mu}(\eta_{t}(k)) = \E_{\mu}(\eta_{0}(k))e^{-t} + \dfrac{N}{K}.
$$
Solving equation \eqref{eq:EDO} ends the proof.
\end{proof}

\subsection{Properties of the invariant measure}
\label{subsec:invariant} 
As $(\eta_t)_{t\geq 0}$ is an irreducible Markov chain on a finite state space, it is straightforward that it admits a unique invariant measure. In fact, this invariant distribution is reversible and we know its expression. 

\begin{theo}[Invariant distribution]
The process $(\eta_{t})_{t\geq 0}$ admits a unique invariant and reversible measure $\nu_N$, which is defined, for every $\eta \in E$, by
\begin{equation*}
\nu_{N}(\{\eta\})= Z^{-1} \prod_{i=1}^{K} \prod_{j=0}^{\eta(i)-1} \dfrac{N-1+Kpj}{j+1},
\end{equation*}
where $Z$ is a normalizing constant.
\end{theo}

This result is a slight generalisation of \cite[Section 4]{GRV10} and \cite[Theorem 2.1]{GRV11}.

\begin{proof}[Proof]
A measure $\nu$ is reversible if and only if it satisfies the following balance equation
\begin{equation}
\label{eq:balance}
\nu(\{\eta\}) C(\eta,\xi) = \nu(\{\xi\}) C(\xi,\eta)\\
\end{equation}
where $ \xi = T_{i \rightarrow j} \eta$ and $C(\eta,\xi) = \mathcal{L} \mathds{1}_{\xi} (\eta) = \eta(i) (K^{-1}+ p\eta(j)(N-1)^{-1})$.

Due to the geometry of the complete graph, it is natural to consider that $\nu$ has the following form 
$$
\nu(\{\eta\})= \frac{1}{Z}\prod_{i=1}^{K} l(\eta(i)),
$$ 
where $l: \{0, \dots, N\} \rightarrow [0,1]$ is a function and $Z$ is a normalizing constant. From \eqref{eq:balance}, we have
\begin{equation*}
l(\eta(i))l(\eta(j))\eta(i) (N-1+ Kp\eta(j)) = l(\eta(i)-1)l(\eta(j)+1)(\eta(j)+1) (N-1+ Kp(\eta(i)-1)),
\end{equation*}
for all $\eta \in E$ and $i,j \in \{1, \dots K \}$. Hence, 
$$
\frac{l(n)}{l(n-1)}\frac{n}{N-1 +Kp(n-1)} = \frac{l(m)}{l(m-1)}\frac{m}{N-1+Kp(m-1)} = u,
$$
for every $m,n\in\{ 1,\dots,N\}$ and some $u \in \R$. Finally,

\begin{eqnarray*}
\nu(\{\eta\})= \prod_{i=1}^{K} \left( u^{\eta(i)} \prod_{j=0}^{\eta(i)-1} \dfrac{N-1+Kpi}{i+1}  l(0)\right) =  l(0)^K u^{N} \prod_{i=1}^{K}  \prod_{j=0}^{\eta(i)-1} \dfrac{N-1+Kpj}{j+1},
\end{eqnarray*}
and $Z= 1/(l(0)^K u^{N})$.
\end{proof}

In particular, we have directly

\begin{coro}[Invariant distribution when $p=1/K$]
If $p=1/K$ then the process $(\eta_{t})_{t\geq 0}$ admits a unique invariant and reversible measure $\nu_N$, which is defined, for every $\eta \in E$, by
\begin{equation*}
\nu_{N}(\{\eta\})= Z^{-1} \prod_{i=1}^{K} \binom {N-2+\eta(i)} {N-2},
\end{equation*}
where $Z$ is a normalizing constant given by
$$
Z= \binom {(K+1)N-K-1} {KN-K-1}.
$$
\end{coro}

\begin{coro}[Marginal laws when $p=1/K$] 
If $p=1/K$ then for all $i\in \left\lbrace 1, \dots ,K\right\rbrace$ we have  
$$
\mathbb{P}_{\nu_{N}}(\eta(i)=x) = \dfrac{1}{Z} \binom {N-2+x} {N-2} \binom {KN-K-x} {(K-1)N-K},
$$
\end{coro}

\begin{proof}[Proof]
Firstly let us recall the Vandermonde binomial convolution type formula: let $n,n_1, \dots, n_p$ be some non-negative integers satisfying $\displaystyle \sum_{i=1}^p n_i = n$,
we have
\begin{equation*}
\binom{r-1}{n-1}= \sum_{r_1 + \dots + r_p= r} \ \prod_{j=1}^p \binom{r_j-1}{n_j-1}.
\end{equation*}
The proof is based on the power series decomposition of $z \mapsto \left(z/(1-z)\right)^{n}=  \displaystyle \prod_{i=1}^p \left(z/(1-z)\right)^{n_{i}}.$
Using this formula, we find
\begin{align*}
\mathbb{P}_{\nu_{N}}(\eta(i)=x)
&= \sum_{\overline{x}\in E_{1}} \mathbb{P}_{\nu_{N}}(\eta=(x_{1}, \dots ,x_{i-1},x,x_{i+1} \dots , x_{K}))\\
&= \dfrac{1}{Z}\binom {N-2+x} {N-2} \sum_{\overline{x}\in E_{1}} \prod_{l=1}^{i-1}\prod_{l=i+1}^{K} \binom {N-2+x_{l}} {N-2}\\
&=\dfrac{1}{Z} \binom {N-2+x} {N-2} \binom {(K-1)(N-1)+N-x-1} {(K-1)(N-1)-1},\\
\end{align*}
where $E_{1} = \left\lbrace \overline{x}=(x_{1}, \dots ,x_{i-1},x_{i+1} \dots , x_{K}) | x_{1}+ \dots +x_{i-1}+x_{i+1} \dots + x_{K} = N-x \right\rbrace$.
\end{proof}

We are now able to express the particle correlations under this invariant measure.

\begin{theo}[Correlation estimates]
\label{theo: cor}
For all $i\neq j \in \lbrace 1, \dots, K\rbrace$, we have
\begin{equation*}
\vert \cov_{\nu_{N}}(\eta(i)/N,\eta(j)/N)\vert  \sim_{N} \dfrac{p+1}{K^{2}N},
\end{equation*}
 
\end{theo}

\begin{proof}[Proof]
Let $\eta$ be a random variable with law $\nu_N$. As $\eta(1), \dots, \eta(K)$ are identically distributed and $\displaystyle \sum_{i=1}^{K} \eta(i) = N$ we have
$$
\cov_{\nu_N}(\eta(i)/N,\eta(j)/N) = -\dfrac{\var_{\nu_N} (\eta(i)/N)}{K-1}.
$$
Using the results of Section \ref{sect:spectre}, we have
\begin{align*}
\mathcal{L}(\eta(i)^{2}) 
&= \eta(i)^{2}\left[-2-\dfrac{2p}{N-1}\right] + \eta(i) \left[\dfrac{2N}{K} + \dfrac{2pN}{N-1} + \dfrac{K-2}{K}\right] + \dfrac{N}{K}.
\end{align*}
Using the fact that $ \displaystyle \int \mathcal{L}(\eta(i)^{2}) d\nu_{N} = 0$ and $ \displaystyle \int \eta(i) d\nu_{N} = \frac{N}{K}$, we deduce that
\begin{align*}
\int \eta(i)^{2} d\nu_{N} &= \dfrac{N \left[(2N+K-2)(N-1) + 2KNp + K(N-1)\right]} {2K^{2}(N-1+p)}.
\end{align*} 

Finally,
\begin{align*}
\var_{\nu_{N}}(\eta(i))
&= \int \eta(i)^{2} d\nu_{N} - \left(\int \eta(i) d\nu_{N}\right)^{2}=\dfrac{N(K-1)(Np+N-1)}{K^{2}(N-1+p)}, 
\end{align*}         
and thus, for $i\neq j$,  
\begin{align*}
\vert \cov_{\nu_{N}}(\eta(i)/N,\eta(j)/N)\vert &\sim_{N} \dfrac{p+1}{K^{2}N}.
\end{align*}               

\end{proof}

\begin{Rq}[Proof through coalescence methods]
Maybe we can use properties of Kingman's coalescent type process (which is a dual process) to recover some of our results (as for instance the previous correlation estimates). Indeed, after an interacting event, all individuals evolve independtly and it is enough to look when the first
mutation happens (backwards in time) on one of the genealogical tree branches. Nevertheless, we prefer to use another approach based on Markovian techniques. 
\end{Rq}

\begin{Rq}[Number of sites]
Theorem \ref{theo: cor} gives the rate of the decay of correlations with respect to the number of particles, but we also have a rate with respect to the number of sites $K$. For instance when $p=1/K$ and if $\eta$ is distributed under the invariant measure, then
\begin{align*}
\vert \cov_{\nu_{N}}(\eta(i)/N,\eta(j)/N)\vert &\sim_{K} \dfrac{1}{K(K-1)N}.
\end{align*}
\end{Rq}

The previous theorem shows that the occupation numbers of two distinct sites become non-correlated when the number of particles increases. In fact, Theorem \ref{theo: cor} leads to a propagation of chaos:

\begin{coro}[Convergence to the QSD]
We have
$$
\E_{\nu_{N}}\left[ d_{\textrm{TV}}(\mu^{N},\pi_{K})\right] \leq \sqrt{\dfrac{K(p+1)}{N}},
$$
where $\pi_{K}$ is the uniform measure on $\{1, \dots, K\}$.
\end{coro}

\begin{proof}[Proof] By the Cauchy-Schwarz inequality, we have
\begin{align*}
\E_{\nu_{N}}\left[\left| \dfrac{\eta(k)}{N} - \dfrac{1}{K}\right| \right]
&\leq \left( \E_{\nu_{N}} \left[\left| \dfrac{\eta(k)}{N} - \dfrac{1}{K}\right|^{2}\right]\right) ^{\frac{1}{2}} = \var_{\nu_{N}}\left(\dfrac{\eta(k)}{N}\right)^{1/2}\leq \sqrt{\dfrac{(K-1)(p+1)}{K^{2}N}}.
\end{align*}
Summing over $\{1, \dots, K\}$ ends the proof.
\end{proof}
Result \cite[Theorem 1.2]{CT15} and its corollaries states that there exist $C,\theta>0$ such that
$$
\E_{\nu_{N}}\left[ d_{\textrm{TV}}(\mu^{N},\pi_{K})\right] \leq \frac{C}{N^\theta}.
$$
All constants are explicit and we have $\theta<1/2$. The last corollary then gives a better bound. To our knowledge, it is the first time that this rate of convergence is obtained for the Fleming-Viot process in discrete space. With spectral arguments, this type of result was obtained for diffusion processes in \cite{MR06}. This bound is achieved because of the absence of bias term. Indeed, 
$$
\forall k \in F^*, \ \E_{\nu_{N}}[\mu^{N}(k)] = \frac{1}{K} = \pi_K(k).
$$ 
The bad term in \cite[Theorem 1.2]{CT15} comes from, with the notations of its proof, the estimation of $|u_k(t)-v_k(t)|$ and Gronwall Lemma. 

\begin{Rq}[Parameters depending on $N$]
\label{rq:dependenceGC}
A nice application of explicit rates of convergence is to consider parameters depending on $N$. For instance, we can now consider that $p=p_N$ depends on $N$, this does not change neither the conditioned semi-goup nor the QSD but this changes the dynamics of our interacting-particle system. The last corollary gives that if $\lim \limits_{N\rightarrow \infty} p_N/N =0$ then the empirical measure converges to the uniform measure.

\end{Rq}

\subsection{Long time behavior and spectral analysis of the generator}
\label{sect:spectre}

In this subsection, we point out the optimality of Theorem \ref{th:tl-general} in this special case. Conditions in Theorem \ref{th:tl-general}, which seems to be a bit strong, are tight in the complete graph dynamics. In that case, $\lambda=\rho=1$ and the bound obtained is optimal in terms of contraction. Moreover, the obtained rate is exactly the spectral gap.

\begin{coro}[Wasserstein contraction]
\label{cor: WassGC}
For any processes $(\eta_{t})_{t>0}$ and $(\eta'_{t})_{t>0}$ generated by \eqref{eq:generator}, and
for any $t \geq 0$, we have
$$
\mathcal{W}_{d} (\mathrm{Law}(\eta_t), \mathrm{Law}(\eta'_t)) \leq e^{- t} \mathcal{W}_{d} (\mathrm{Law}(\eta_0), \mathrm{Law}(\eta'_0)).
$$
In particular, when $(\eta'_{0})$ follows the invariant distribution $\nu_{N}$ associated to \eqref{eq:generator}, we
get for every $t \geq 0$
$$
\mathcal{W}_{d} (\mathrm{Law}(\eta_t), \nu_{N}) \leq e^{- t} \mathcal{W}_{d} (\mathrm{Law}(\eta_0), \nu_{N}).
$$
\end{coro}

In particular, if $\lambda_1$ is the smallest positive eigenvalue of $-\mathcal{L}$, defined at \eqref{eq:generator}, then we have 
$$1=\rho \leq \lambda_1.$$
Indeed, on the one hand, let us recall that, as the invariant measure is reversible, $\lambda_1$ is the largest constant such that
\begin{equation}
\label{eq:var-zero}
\lim_{t\rightarrow + \infty} e^{2 \lambda t} \Vert R_t f - \nu_{N}(f) \Vert_{L^{2}(\nu_{N})}^{2} =0,
\end{equation}
for every $\lambda < \lambda_1$ and $f\in L^{2}(\nu_{N})$, where $(R_t)_{t\geq 0}$ is the semi-group generated by $\mathcal{L}$. See for instance \cite{B94,S97}. On the other hand, if $\lambda < 1$ then, by Theorem \ref{th:tl-general}, we have
\begin{align*}
e^{2 \lambda t} \Vert R_t f - \nu_{N}(f) \Vert_{L^{2}(\nu_{N})}^{2}
&=e^{2 \lambda t} \int_E \left( (\delta_\eta R_t) f - (\nu_{N} R_t) f\right)^2 \nu_{N}(d\eta)\\
&\leq 2 e^{2 \lambda t} \Vert f \Vert_\infty^2 \int_E  \mathcal{W}_{d} (\delta_\eta R_t, \nu_{N} R_t)^2 \nu_{N}(d\eta)\\
&\leq 2 e^{2 (\lambda-1) t} \Vert f \Vert_\infty^2 \int_E \mathcal{W}_{d} (\delta_\eta, \nu_{N})^2 \nu_{N}(d\eta),
\end{align*}
and then \eqref{eq:var-zero} holds. Now, the constant functions are trivially eigenvectors of $\mathcal{L}$ associated with the eigenvalue $0$, and if, for $k \in \{1,\dots,K\}$, $l \geq 1$ we set $\varphi^{(l)}_{k}:\eta\mapsto \eta(k)^{l}$ then the function $\varphi^{(1)}_k$ satisfies 
$$
\mathcal{L}\varphi^{(1)}_k= N/K - \varphi^{(1)}_k.
$$
In particular $\varphi^{(1)}_k - N/K$ is an eigenvector and $1$ is an eigenvalue of $-\mathcal{L}$. This gives $\lambda_1 \leq 1$ and finally $\lambda_1 =1$ is the smallest eigenvalue of $-\mathcal{L}$. By the reversibility, we have a Poincar\'{e} (or spectral gap) inequality
$$
\forall t \geq 0, \ \Vert R_t f - \nu_{N}(f) \Vert_{L^{2}(\nu_N)}^{2} \leq e^{-2 t} \Vert f - \nu_{N}(f) \Vert_{L^{2}(\nu_{N})}^{2}.
$$

\begin{Rq}[Complete graph random walk]
If $(a_i)_{1 \leq i \leq K}$ is a sequence such that $\displaystyle \sum_{i=1}^{K} a_{i}=0$ then the function $\displaystyle \sum_{i=1}^{K} \varphi^{(1)}_i$ is an eigenvector of $\mathcal{L}$. However, if $L$ is the generator of the classical complete graph random walk, $La = - a$ and then $a$ is also an eigenvector of $L$ with the same eigenvalue.
\end{Rq}

Let us finally give the following result on the spectrum of $\mathcal{L}$:

\begin{lem}[Spectrum of $-\mathcal{L}$]
The spectrum of $-\mathcal{L}$ is included in
$$
\left\{ \sum_{i=1}^{K} \lambda_{l_{i}} \ | \ l_1, \dots, l_K \in \{0, \dots, N\} \right\},
$$
where 
$$
\forall l \in \{0, \dots, N\}, \ \lambda_{l} = l + \dfrac{l(l-1)p}{N-1}.
$$
\end{lem}

\begin{proof}[Proof]

For  every $ k \in \{1, \dots, K\}$ and $l\in \{0, \dots, N\}$, we have
\begin{align*}
\mathcal{L}\varphi^{(l)}_{k}(\eta) 
&= - \lambda_{l} \varphi_{k}^{(l)}(\eta) + Q_{l-1}(\eta),
\end{align*}
where $Q_{l-1}$ is a polynomial whose degree is less than $l-1$. A straightforward recurrence shows that whether there exists or not a polynomial function $\psi^{(l)}_{k}$, whose degree is $l$, satisfying $\mathcal{L}\psi^{(l)}_{k} = - \lambda_l \psi^{(l)}_{k}$ (namely $\psi^{(l)}_{k}$ is an eigenvector of $\mathcal{L}$). Indeed, it is possible to have $\psi^{(l)}_{k} = 0$ since the polynomial functions are not linearly independent ($F$ is finite). More generally, for all  $l_1, \dots, l_K \in \{1, \dots, N\}$, there exists a polynomial $Q$ with $K$ variables, whose degree with respect to the $i^{\text{th}}$ variable is strictly less than $l_i$, such that the function $\phi: \eta \mapsto \displaystyle \prod_{i=1}^{K} \eta(k_{i})^{l_{i}} + Q(\eta)$ satisfies 
$$
\mathcal{L}\phi = - \lambda \phi \ \text{where} \ \lambda  = \sum_{i=1}^{K} \lambda_{l_{i}}.
$$ 
Again, provided that $\phi \neq 0$, $\phi$ is an eigenvector and $\lambda$ an eigenvalue of $-\mathcal{L}$. Finally, as the state space is finite, using multivariate Lagrange polynomial, we can prove that every function is polynomial and thus we capture all the eigenvalues.
\end{proof}

\begin{Rq}[Cardinal of $E$]
As $\textrm{card}(F^*)=K$, we have 
$$
\textrm{card}(E) = \binom{N+K-1}{K-1} = \frac{(N+K-1)!}{N!(K-1)!}.
$$
In particular, the number of eigenvalues is finite and less than $\textrm{card}(E)$.
\end{Rq}

\begin{Rq}[Marginals]
For each $k$, the random process $(\eta_t(k))_{t\geq0}$, which is a marginal of a process generated by \eqref{eq:generator}, is a Markov process on $\N_{\textrm{N}}=\{0,\dots, N\}$ generated by 
\begin{align*}
\mathcal{G} f(x) 
&= (N-x) \left(\frac{1}{K}+\dfrac{px}{N-1}\right) (f(x+1) - f(x))\\
&\quad + x \left( \frac{K-1}{K} + \frac{p(N- x)}{N-1}\right)(f(x-1) - f(x)),
\end{align*}
for every function $f$ on $\N_{\textrm{N}}$ and $x\in \N_{\textrm{N}}$. We can express the spectrum of this generator. Indeed, let $\varphi_l: x \mapsto x^l$, for every $l\geq 0$. The family $(\varphi_l)_{0 \leq l\leq N}$ is linearly independent as can be checked with a Vandermonde determinant. This family generates the $L^2-$space  associated to the invariant measure since this space has a dimension equal to $N+1$. Now, similarly to the proof of the previous lemma, we can prove the existence of $N+1$ polynomials, which are eigenvectors and linearly independent, whose eigenvalues are $\lambda_0, \lambda_1, \dots, \lambda_N$.
\end{Rq}

\section{The two point space}
\label{sect:two}
In all this section we denote by $p_0$ the function $i \in F^{*} \mapsto Q_{i,0}$.\\
We consider a Markov chain defined on the states $\left\lbrace 0, 1, 2\right\rbrace$ where $0$ is the absorbing state. Its infinitesimal generator $G$ is defined by
$$
G=\begin{bmatrix}
   0 & 0 & 0 \\
   p_0(1) & -a - p_0(1) & a\\
   p_0(2) & b & -b - p_0(b), \\
\end{bmatrix}
$$ 
where $a,b >0$, $p_0(1),p_0(2)\geq 0$ and $p_0(1)+p_0(2)>0$. The generator of the Fleming-Viot process with $N$ particles applied to bounded functions $f: E\rightarrow \mathbb{R}$ reads 
\begin{align}
\label{eq:generator2point} 
\mathcal{L}f(\eta) 
&= \eta(1) \left( a + p_{0}(1) \dfrac{\eta(2)}{N-1} \right)  (f(T_{1 \rightarrow 2} \eta)-f(\eta))\nonumber \\
&+ \eta(2) \left( b + p_{0}(2) \dfrac{\eta(1)}{N-1} \right)  (f(T_{2 \rightarrow 1} \eta)-f(\eta)).
\end{align}

\subsection{The associated killed process}

The long time behavior of the conditionned process is related to the eigenvalues and eigenvectors of the matrix:
$$
M=
\begin{bmatrix}
   -a-p_{0}(1) & a \\
   b & -b-p_{0}(2) 
\end{bmatrix}.
$$
Indeed see \cite[section 3.1]{MV12}. Its eigenvalues are given by      
\begin{equation*}
\lambda_{+} = \dfrac{-(a+b+p_{0}(1)+p_{0}(2)) +\sqrt{(a-b+p_{0}(1)- p_{0}(2))^{2}+4ab}}{2},
\end{equation*}

\begin{equation*}
\lambda_{-} = \dfrac{-(a+b+p_{0}(1)+p_{0}(2)) -\sqrt{(a-b+p_{0}(1)- p_{0}(2))^{2}+4ab}}{2},
\end{equation*}

and the corresponding eigenvectors are respectively given by
$$  v_{+} = \left (
   \begin{array}{c}
      a \\
      -A+\sqrt{A^{2}+4ab}\\
   \end{array}
   \right ) \ \text{ and } \ v_{-} = \left (
                                        \begin{array}{c}
                                          a \\
                                          -A-\sqrt{A^{2}+4ab}\\
                                        \end{array}
                                        \right ),$$
   
where $A= a-b+p_{0}(1)- p_{0}(2)$. Also set $\nu= v_+/(v_+(1) + v_+(2))$. From these properties, we deduce that

\begin{lem}[Convergence to the QSD]
\label{lem:cvQSD2pt}
There exists a constant $C>0$ such that for every initial distribution $\mu$, we have
$$
\forall t\geq 0, \ d_{\text{\textrm{TV}}}(\mu T_t, \nu) \leq C e^{-(\lambda_+ - \lambda_-) t}.
$$
\end{lem}
\begin{proof}[Proof]
See \cite[Theorem 7]{MV12} and \cite[Remark 3]{MV12}.
\end{proof}
Note that
$$
\lambda_+ - \lambda_-
=\sqrt{(a+b)^2 +2 (a-b)(p_{0}(1)- p_{0}(2))+(p_{0}(1)- p_{0}(2))^{2}}$$
and $\rho= a+b - (\sup(p_0)-\inf(p_0))$ defined in Theorem \ref{th:tl-general}. We have then $\lambda_+ - \lambda_-> \rho$ when $\sup(p_0)>\inf(p_0)$. In particular Theorem \ref{th:tl-general} seems not optimal.

\subsection{Explicit formula of the invariant distribution}
\label{sect:two-inv} 
Firstly note that, as
$$
\forall \eta \in E, \eta(1) + \eta(2) =N,
$$
each marginal of $(\eta_t)_{t\geq 0}$ is a Markov process:

\begin{lem}[Markovian marginals]
The random process $(\eta_t(1))_{t\geq0}$, which is a marginal of a process generated by \eqref{eq:generator2point}, is a Markov process generated by  $\mathcal{G}$ defined by
\begin{align}
\mathcal{G} f(n) 
&= b_n (f(n+1) - f(n)) + d_n (f(n-1) - f(n)), \label{eq:gen-G}
\end{align}
for any function $f$ and $n \in \N_{\mathrm{N}}=\{0, \dots, N\}$, where
$$
b_n=(N-n) \left( b + p_{0}(2) \dfrac{n}{N-1} \right)\text{ and } d_n = n \left( a + p_{0}(1) \dfrac{N-n}{N-1} \right).
$$
\end{lem}

\begin{proof}[Proof]
For every $\eta\in E$, we have $\eta=(\eta(1), N-\eta(1))$ thus the Markov property and the generator are easily deducible from the properties of $(\eta_t)_{t \geq 0}$. 
\end{proof}

From this result and the already known results on birth and death processes \cite{CJ12, C04}, we deduce that $(\eta_t(1))_{t\geq0}$ admits an invariant and reversible distribution $\pi$ given by
$$
\pi(n) = u_0 \prod_{k=1}^n \frac{b_{k-1}}{d_k} \ \text{ and } \ u_0^{-1}= 1 + \sum_{k=1}^N \frac{b_0 \cdots b_{k-1}}{d_{1} \cdots d_{k} },
$$ 
for every $n\in \N_{\mathrm{N}}$. This gives
$$
\pi(n) = u_0 \binom{N}{n} \prod_{k=1}^n \frac{b(N-1) + (k-1) p_0(2)}{ a(N-1) + (N-k) p_0(1)},
$$
and 
$$
u_0^{-1} = 1 + \prod_{k=1}^N \frac{b(N-1) + k p_0(2)}{ a(N-1) + k p_0(1)}.
$$
Similarly, as $\eta_t(2) = N-\eta_t(1)$, the process $(\eta_t(2))_{t\geq 0}$ is a Markov process whose invariant distribution is also easily calculable. The invariant law of $(\eta_t)_{t\geq 0}$, is then given by
$$
\nu_N((r_1,r_2))=\pi\left(\{ r_1\}\right),\quad \forall (r_1,r_2) \in E.$$
Note that if $p_0$ is not constant then we can not find a basis of orthogonal polynomials in the $L^2$ space associated to $\nu_N$. It is then very difficult to express the spectral gap or the decay rate of the correlations.

\subsection{Rate of convergence}

Applying Theorem \ref{th:tl-general}, in this special case, we find:

\begin{coro}[Wasserstein contraction]
For any processes $(\eta_{t})_{t>0}$ and $(\eta'_{t})_{t>0}$ generated by \eqref{eq:generator2point}, and
for any $t \geq 0$, we have
\begin{align*}
\mathcal{W}_{d} (\mathrm{Law}(\eta_t), \mathrm{Law}(\eta'_t)) \leq e^{-\rho t} \mathcal{W}_{d} (\mathrm{Law}(\eta_0), \mathrm{Law}(\eta'_0)),
\end{align*}
where $\rho = a+b-(\sup(p_0)-\inf(p_0)).$ In particular, when $(\eta'_{0})$ follows the invariant distribution $\nu_N$ of \eqref{eq:generator2point}, we
get for every $t > 0$
$$
\mathcal{W}_{d} (\mathrm{Law}(\eta_t), \nu_N) \leq e^{-\rho t} \mathcal{W}_{d} (\mathrm{Law}(\eta_0), \nu_N).
$$
\end{coro}

This result is not optimal. Nevertheless, the error does not come from the coupling choice of \cite{CT15} but it comes from how the distance is estimated. Indeed, this coupling induces a coupling between two processes generated by $\mathcal{G}$ defined by \eqref{eq:gen-G}. More precisely, let $\mathbb{L} = \mathbb{L}_Q + \mathbb{L}_{p}$ be the generator of the coupling introduced in the proof of \cite[Theorem 1.1]{CT15} in this special case. We set $\mathbb{G} =\mathbb{G}_Q + \mathbb{G}_{p}$, where for any $n,n'\in \N_N$ and $f$ on $E \times E$,
$$
\mathbb{L}_Q f((n,N-n),(n',N-n'))= \mathbb{G}_Q \varphi_f(n,n'),
$$
$$
\mathbb{L}_p f((n,N-n),(n',N-n'))= \mathbb{G}_p \varphi_f(n,n'),
$$
and $\varphi_f(n,n') =f((n,N-n),(n',N-n'))$. It satisfies, for any function $f$ on $\N_{\mathrm{N}}$ and $n'> n$ two elements of $\N_{\mathrm{N}}$,

\begin{align*}
\mathbb{G}_{Q} f(n,n')
&= na \left( f(n-1,n'-1) - f(n,n') \right)\\
&+ (N-n')b \left( f(n+1,n'+1) - f(n,n') \right)\\
&+ (n'-n)b \left( f(n+1,n') - f(n,n') \right)\\
&+ (n'-n)a \left( f(n,n'-1) - f(n,n') \right),
\end{align*}
and
\begin{align*}
\mathbb{G}_{p} f(n,n')
&= p_0(1) \frac{n (N-n')}{N-1} \left( f(n-1,n'-1) - f(n,n') \right)\\
&+ p_0(2) \frac{n (N-n')}{N-1} \left( f(n+1,n'+1) - f(n,n') \right)\\
&+ p_0(1) \frac{n (n'-n)}{N-1} \left( f(n-1,n') - f(n,n') \right)\\
&+ p_0(2) \frac{(N-n')(n'-n)}{N-1} \left( f(n,n'+1) - f(n,n') \right)\\
&+ p_0(2) \frac{n(n'-n)}{N-1} \left( f(n+1,n') - f(n,n') \right)\\
&+ p_0(1) \frac{(N-n')(n'-n)}{N-1} \left( f(n,n'-1) - f(n,n') \right).
\end{align*}
Now, for any sequence of positive numbers $(u_k)_{k\in \{0, \dots, N-1\}}$, we introduce the distance $\delta_u$ defined by
$$
\delta_u (n,n') = \sum_{k=n}^{n'-1} u_k,
$$
for every $n,n' \in \N_{\mathrm{N}}$ such that $n'>n$. For all $n\in \N_{\mathrm{N}}\backslash \{N\}$, we have $\mathbb{G} \delta_u(n,n+1) \leq - \lambda_u \delta_u(n,n+1)$ where
$$
\lambda_u = \min_{k\in \{0, \dots, N-1\}} \left[ d_{k+1} - d_k \frac{u_{k-1}}{u_k} + b_k - b_{k+1} \frac{u_{k+1}}{u_{k}}\right],
$$ 
and thus, by linearity, $\mathbb{G} \delta_u(n,n') \leq - \lambda_u \delta_u(n,n')$, for every $n,n' \in \N_{\mathrm{N}}$. This implies that for any processes $(X_{t})_{t\geq 0}$ and $(X'_{t})_{t\geq 0}$ generated by $\mathcal{G}$ , and for any $t \geq 0$,
$$
\mathcal{W}_{\delta_u} (\mathrm{Law}(X_t), \mathrm{Law}(X'_t)) \leq e^{-\lambda_u t} \mathcal{W}_{\delta_u} (\mathrm{Law}(X_0), \mathrm{Law}(X'_0)).
$$
Note that, for every $n,n' \in \N_{\mathrm{N}}$, we have
$$
\min(u) d((n,N-n),(n',N-n')) \leq \delta_u(n,n') \leq \max(u) d((n,N-n),(n',N-n')),
$$
and then for any processes $(\eta_{t})_{t\geq 0}$ and $(\eta'_{t})_{t\geq 0}$ generated by \eqref{eq:generator2point}, and for any $t \geq 0$, we have
$$
\mathcal{W}_{d} (\mathrm{Law}(\eta_t), \mathrm{Law}(\eta'_t)) \leq \frac{\max(u)}{\min(u)} e^{-\lambda_u t} \mathcal{W}_{d} (\mathrm{Law}(\eta_0), \mathrm{Law}(\eta'_0)).
$$
Finally, using \cite[Theorem 9.25]{C04}, there exists a positive sequence $v$ such that $\lambda_v = \displaystyle \max_u \lambda_u>0$ is the spectral gap of the birth and death process $(\eta_t(1))_{t\geq 0}$. These parameters depend on $N$ and so we should write the previous inequality as
\begin{equation}
\label{eq:gap-2pts}
\mathcal{W}_{d} (\mathrm{Law}(\eta_t), \mathrm{Law}(\eta'_t)) \leq C(N) e^{-\lambda_N t} \mathcal{W}_{d} (\mathrm{Law}(\eta_0), \mathrm{Law}(\eta'_0)),
\end{equation} 
where $C(N)$ and $\lambda_N$ are two constants depending on $N$. In conclusion, the coupling introduced in Theorem \ref{th:tl-general} gives the optimal rate of convergence but we are not able to express a precise expression of $\lambda_N$ and $C(N)$. Nevertheless, in the section that follows, we will prove that, whatever the value of the parameters, the spectral gap is always bounded from below by a positive constant not depending on $N$.

\subsection{A lower bound for the spectral gap}
In this subsection, we study the evolution of $(\lambda_N)_{N\geq 0}$. Calculating $\lambda_N$ for small value of $N$ (it is the eigenvalue of a small matrix) and some different parameters show that, in general, this sequence is not monotone and seems to converge to $\lambda_+ - \lambda_-$. We are not able to prove this, but as it is trivial that for all $N\geq 0$, $\lambda_N>0$, we can hope that it is bounded from below. The aim of this section is to prove this fact.

Firstly, using similar arguments of subsection \ref{sect:spectre}, we have $\lambda_N \geq \rho$, for every $N\geq 0$. This result does not give us information in the case $\rho \leq 0$. However, we can use Hardy's inequalities \cite[Chapter 6]{Toulouse} and mimic some arguments of \cite{M99} to obtain:

\begin{theo}[A lower bound for the spectral gap]
\label{th:spectral gap2}
If $\rho \leq 0$ then there exists $c>0$ such that  
$$
\forall N \geq 0, \ \lambda_N > c.
$$  
\end{theo}

The rest of this subsection aims to prove this result. Hardy's inequalities are mainly based on the estimation of the quantities $B_{N,+}$ and $B_{N,-}$ defined for every $i\in \N$ by
\begin{equation}
\label{eq:def-B}
 B_{N,+} (i) = \max_{x>i}\left( \sum_{y=i+1}^{x} \dfrac{1}{\pi(y) d_y}\right) \pi([x,N]),
\end{equation}
and
$$
 B_{N,-}(i) = \max_{x<i}\left( \sum_{y=x}^{i-1} \dfrac{1}{\pi(y) b_y}\right) \pi([1,x]).
$$

We recall that $\pi=\pi_N$ is the invariant distribution defined in Subsection \ref{sect:two-inv} and jumps rates $b$ and $d$ also depend on $N$.

 More precisely, \cite[Proposition 3]{M99} shows that if one wants to get a "good" lower bound of the spectral gap, one only needs to guess an "adequate choice" of $i$ and to apply the estimate
$$
\lambda_N \geq \frac{1}{4\max\{B_{N,+}(i) , B_{N,-}(i)\}}.
$$
So, we have to find an upper bound for these two quantities. Before to give it, let us prove that the invariant distribution $\pi$ is unimodal. Indeed, it will help us to choose an appropriate $i$.

\begin{lem}[Unimodality of $\pi$]
\label{lem: decreasing}
The sequence $(\pi(i+1)/\pi(i))_{i\geq 0}$ is decreasing. 
\end{lem}


\begin{proof}[Proof of Lemma \ref{lem: decreasing}]
For all $i \in \{1, \dots, N\}$, we set 
$$
g(i)= \dfrac{\pi(i+1)}{\pi(i)}=\dfrac{(N-i)(b(N-1)+ip_0(2))}{(i+1)((a+ p_0(1))(N-1)-ip_0(1))}.
$$
It follows that 

\begin{align*}
g(i+1)-g(i)  
&= \dfrac{\Lambda_N(i)}{(i+1)((a+ p_0(1))(N-1)-i p_0(1))(i+2)((a+ p_0(1))(N-1)-(i+1) p_0(1))}
\end{align*}

where 
\begin{align*}
\Lambda_N(i)
&= (N-i-1)(b(N-1)+(i+1)p_0(2))(i+1)((a+ p_0(1))(N-1)-i p_0(1)) \\
&\quad \ - (N-i)(b(N-1)+ip_0(2))(i+2)((a+ p_0(1))(N-1)-(i+1) p_0(1))\\
&= - \left[ b(N-1)- p_0(2) \right] \left[ (N+1) \left(a(N-1)- p_0(1) \right) + p_0(1) (N-i)(N-i-1) \right]\\
&\quad  \ - p_0(2) \left(i^{2} + 3i +2 \right) \left( a(N-1) - p_0(1) \right)\\
& \leq 0.
\end{align*}

We deduce the result.
\end{proof}


\begin{proof}[Proof of Theorem \ref{th:spectral gap2}]
Without less of generality, we assume that $p_0(1)\geq  p_0(2)$ and we recall that $\rho \leq 0$. We would like to know where $\pi$ reaches its maximum $i^*$ since it will be a good candidate to estimate $B_{N,+}(i^*)$ and $B_{N,-}(i^*)$. From the previous lemma, to find it, we look when $\pi(i+1)/\pi(i)$ is close to one. We have, for all $i\in \{ 1, \dots, N\}$,

\begin{align}
\dfrac{\pi(i+1)}{\pi(i)} 
&= \dfrac{b_i}{d_{i+1}}=1 + \dfrac{(p_0(1)-p_0(2))(i-i_1)(i-i_2)}{(i+1)\left((a+p_0(1))(N-1)-ip_0(1)\right)}\label{expression},
\end{align}
where $i_1$ and $i_2$ are the two real numbers given by 
$$
i_1 = \dfrac{ N (a+b+ p_0(1)-p_0(2)) - (a+b+ 2p_0(1)) - \sqrt{\Delta}}{2(p_0(1)-p_0(2))}
$$
and
$$
i_2 = \dfrac{ N (a+b+ p_0(1)-p_0(2)) - (a+b+ 2p_0(1)) + \sqrt{\Delta}}{2(p_0(1)-p_0(2))},
$$
where 
\begin{align*}
\Delta &=[N(a+b+p_0(1) - p_0(2)) - (a+b+ 2p_0(1))]^2 \\
&\quad  \ - 4 (N-1)(bN-a - p_0(1))(p_0(1)-p_0(2)).
\end{align*}
In particular,  $1\leq i_1 \leq N \leq i_2$. Furthermore, if $\lfloor . \rfloor$ denotes the integer part then
$$
\frac{\pi(\lfloor i_1 \rfloor + 2 )}{\pi(\lfloor i_1 \rfloor +1)}  \leq 1 \leq \frac{\pi(\lfloor i_1 \rfloor +1)}{\pi(\lfloor i_1 \rfloor )}.
$$

Let us define $m_N = \lfloor i_1 \rfloor +1$ and $l_N= 2(\lfloor \sqrt{N} \rfloor +1)$. Using a telescopic product, we have
$$
\frac{\pi(m_N+l_N)}{\pi(m_N)} = \frac{\pi(m_N+l_N-\lfloor \sqrt{N} \rfloor -1)}{\pi(m_N)} \prod_{j=1}^{\lfloor \sqrt{N} \rfloor +1} \frac{\pi(m_N+l_N-j + 1)}{\pi(m_N+l_N-j)},
$$
Using Lemma \ref{lem: decreasing} and the previous calculus, we have that the sequences $(\pi(i))_{i \geq m_N}$ and $(\pi(i+1)/\pi(i))_{i \geq 0}$ are decreasing and then
$$
\frac{\pi(m_N+l_N)}{\pi(m_N)} \leq \left( \frac{\pi(m_N+l_N- \lfloor \sqrt{N} \rfloor )}{\pi(m_N+l_N-\lfloor \sqrt{N} \rfloor  -1)} \right)^{\lfloor \sqrt{N} \rfloor +1}.
$$

Now using \eqref{expression} and some equivalents, there exists a constant $\delta_1>0$ (not depending on $N$) such that
\begin{align*}
\frac{\pi(m_N+l_N- \lfloor \sqrt{N} \rfloor )}{\pi(m_N+l_N-\lfloor \sqrt{N} \rfloor  -1)}
&\leq 1- \dfrac{\delta_1}{\sqrt{N}}.
\end{align*}
Using the fact that $1-x \leq e^{-x}$ for all $x \geq 0$, we finally obtain $\pi(m_N+l_N)/\pi(m_N) \leq e^{-\delta_1}$.
Similar arguments entail the existence of $\delta_2>0$ (also not depending on $N$) such that $\pi(m_N-l_N)/\pi(m_N) \leq e^{-\delta_2}$.
In conclusion, using Lemma \ref{lem: decreasing}, we have shown  that for all $i \geq m_N$ and $j \leq m_N$, the following inequalities holds:
$$
\pi(i+l_N) \leq e^{-\delta_1} \pi(i) \ \text{ and } \  \pi(j-l_N) \leq e^{-\delta_2} \pi(j).
$$
We are now armed to evaluate $B_{N,+}(m_N)$ defined in \eqref{eq:def-B}. Firstly, using the expressions of the death rate $d$ and $m_N$, there exist $\gamma>0$ (not depending on $N$) and $N_0 \geq 0$ such that for all $N \geq N_0$ and all $i \geq m_N+1 $, $d_i \geq \gamma N$. Let us fix $x\geq m_N+1$, using that $(\pi(i))_{i\geq m_N}$ is decreasing,  we have
\begin{align*}
\sum_{y=m_N+1}^{x} \frac{1}{\pi(y)}
&= \sum_{\{i,k | m_N +1 \leq k-i l_N \leq x\}}  \frac{1}{\pi(k-i l_N)} \\ 
& \leq \sum_{\{i,k | m_N +1 \leq k-i l_N \leq x\}} \frac{e^{-\delta_1 i}}{\pi(k)} \\
& \leq \dfrac{1}{1-  e^{-\delta_1 }}  \sum_{k=x-l_N +1}^{x} \dfrac{1}{\pi(k)}\\ 
& \leq \dfrac{l_N}{\pi(x)} \dfrac{1}{1-  e^{-\delta_1}}.
\end{align*}
Similarly, we have
$$
\pi([x,N]) =\sum_{\{k,i | x\leq k+i l _N \leq N \}} \mathds{1}_{\{x+il_N \leq N\}} \Pi_N(k+il_N) \leq \frac{l_N \pi(x) }{1-  e^{-\delta_1}}.
$$
Using these three estimates, we deduce that, for every $N\geq N_0$, 
$$
B_{N,+}(m_N) \leq \frac{1}{\gamma N} \left( \dfrac{l_N}{1-  e^{-\delta_1}}\right)^{2} \leq \frac{1}{\gamma N} \left( \dfrac{2(\sqrt{N}+1)}{1-  e^{-\delta_1}}\right)^{2} \leq \frac{16}{\gamma (1-e^{-\delta_1})}. 
$$
The study of $B_{N,-}(m_N)$ is similar.
\end{proof}

\subsection{Simulation and evolution of the spectral gap (of the Fleming-Viot process)}

\label{sect:graph}

\begin{figure}[H]
\includegraphics[scale=0.5]{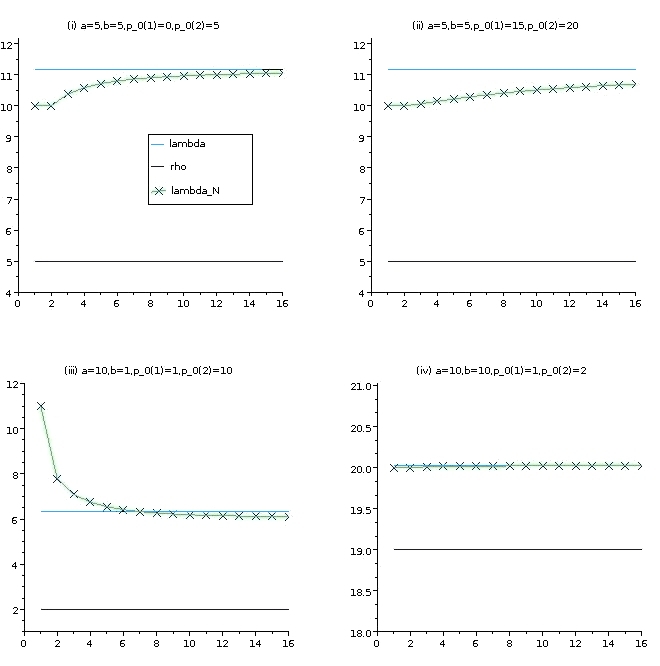}
\caption{Evolution of the spectral gap with respect to the number of particles. Details are described in Subsection \ref{sect:graph}}
\label{fig:2pt}
\end{figure}

As stated in Lemma \ref{lem:cvQSD2pt}, the rate of convergence $\lambda$ of the conditioned semi-group to the QSD is explicit. For a fixed $N\geq 1$, the Fleming-Viot process is a reversible Markov chain (for this example) on a finite space, it then converges to equilibrium at rate $\lambda_N$, where $\lambda_N$ is the eigenvalue of its generator closer to $0$.

Theorem \ref{th:spectral gap2} shows that the rate of convergence to equilibrium of the Fleming-Viot process is uniformly bounded. Nevertheless, there is some natural questions such as the convergence of $\lambda_N$ to $\lambda$ or the monotonicity of this sequence. Also, do we have $\lambda_N \leq \lambda$ or $\lambda_N\geq \lambda$?

When there are few particles, the generator of the Fleming-Viot process is a square matrix of size $2N$ and one can explicitly calculate its spectrum (with the help of a computer). In figure \ref{fig:2pt}, each graphic represents, with different parameters, the evolution of the spectral gap of the conditioned process $\lambda=\lambda_+-\lambda_-$ (detailed in Lemma \ref{lem:cvQSD2pt}), of the Fleming Viot particle system $\lambda_N$ and of the upper-bound $\rho$ in Theorem \ref{th:tl-general}, with respect to the number of particles $N$.

Graphics $(i)$ and $(ii)$ illustrate that when $\text{osc}(p_0)=p_0(2)-p_0(1)$ remains constant for two parameter choices then the Fleming-Viot dynamics are different although the conditioned semi-groups are the same. More $p_0(1)$ is large more jumps there are; these graphics seem to show that the rate of convergence  is dragged down by the interactions. Graphic $(iii)$ shows that $\lambda_N$ is neither increasing nor decreasing and neither upper nor lower than $\lambda$. Graphic $(iv)$ shows that when $\text{osc}(p_0)$ is small then all curves are close.

In any case, it seems that $\lambda_N$ converges to $\lambda$ but this point remains an open question.

\subsection{Correlations}

Using \cite[Theorem 2.6]{CT15}, we have
\begin{coro}[Correlations]
If $(\eta_t)_{t\geq 0}$ is a process generated by \eqref{eq:generator2point} then we have for all $t\geq 0$,
$$
\cov(\eta_{t}(k)/N,\eta_{t}(l)/N) \leq \frac{2}{N^2} \frac{1 - e^{- 2\rho t}}{\rho} \left(N (a \vee b) +  \sup(p_0) \frac{N^2}{N-1}\right).
$$
\end{coro}
If $\rho \leq 0$, the right-hand side of the previous inequality explodes as $t$ tends to infinity whereas these correlations are bounded by $1$. Nevertheless, using Theorem \cite[Theorem 2.6]{CT15}, Remark \cite[Remark 2.7]{CT15} and Inequality \eqref{eq:gap-2pts}, we can prove that there exists two constants $C'(N)$, depending on $N$, and $K$, which does not depend on $N$, such that
$$
\sup_{t\geq 0} \cov(\eta_{t}(k)/N,\eta_{t}(l)/N) \leq C'(N)= \frac{K C(N)}{N \lambda_N},
$$
where $C(N)$ is defined in \eqref{eq:gap-2pts}. Even if Theorem \ref{th:spectral gap2} gives an estimate of $\lambda_N$, $C(N)$ is not (completely) explicit and we do not know if the right-hand side of the previous expression tends to $0$ as $N$ tends to infinity. This example shows the difficulty of finding explicit and optimal rates of the convergence towards equilibrium and the decay of correlations.

\bibliographystyle{abbrv} 

\end{document}